\documentclass[12pt]{amsart}
\usepackage{amsmath,amssymb,amsthm}

\usepackage{ifthen,verbatim}
\usepackage{mathrsfs}
\usepackage{color}
\usepackage{url}

\usepackage[english]{babel}
\usepackage{verbatim,here}
\usepackage[T1]{fontenc}
\usepackage{floatflt,graphicx,graphics}
\usepackage{a4wide}
\numberwithin{equation}{section}
\nonstopmode
\theoremstyle{plain}

\newtheorem{corollary}[equation]{Corollary}

\newtheorem{theorem}[equation]{Theorem}

\newtheorem{lemma}[equation]{Lemma}

\newtheorem{proposition}[equation]{Proposition}

\theoremstyle{definition}
\newtheorem{rem}[equation]{Remark}
\newtheorem{remark}[equation]{Remark}

\newtheorem{example}[equation]{Example}

\newtheorem{nonsec}[equation]{}

{\qed\bigskip}

\newcounter{alphabet}

\newcommand{\be}{\begin{eqnarray}}
\newcommand{\ee}{\end{eqnarray}}
\newcommand{\ba}{\begin{array}}
\newcommand{\ea}{\end{array}}
\newcommand{\ben}{\begin{eqnarray*}}
\newcommand{\een}{\end{eqnarray*}}

\newcommand{\C}{{\mathbb C}}
\newcommand{\R}{{\mathbb R}}

\newcommand{\B}{\mathbb{B}}

\newcommand{\capa}{\mathrm{cap}\,}

\newcommand{{\tth}}{\mathrm{th}}

\newcommand{\K}{\mathcal{K}}

\newcommand{\tnh}{\mathrm{th}}

\newcommand {\M} {\mathsf{M}}

\renewcommand{\i}{\mathrm{i}}
\newcommand{\bs}{{\bf s}}

\font\fFt=eusm10 %scaled 1200
\font\fFa=eusm7  %scaled 1200
\font\fFp=eusm5  %scaled 1200
\def\K{\mathchoice
{\hbox{\,\fFt K}}
{\hbox{\,\fFt K}}
{\hbox{\,\fFa K}}
{\hbox{\,\fFp K}}}
\newcounter{minutes}
\divide\time by 60
\newcounter{hours}
\multiply\time by 60
\addtocounter{minutes}{-\time}
\begin{document}

\bibliographystyle{amsplain}
\title%[]
{
Condenser Capacity and Hyperbolic Diameter
}

\def\thefootnote{}
\footnotetext{
\texttt{\tiny File:~\jobname .tex,
          printed: \number\year-\number\month-\number\day,
          \thehours.\ifnum\theminutes<10{0}\fi\theminutes}
}
\makeatletter\def\thefootnote{\@arabic\c@footnote}\makeatother

\author[M.M.S. Nasser]{Mohamed M. S. Nasser}
\address{Mathematics Program, Department of Mathematics, Statistics and Physics, College of Arts and Sciences, Qatar University, 2713, Doha, Qatar}
\email{mms.nasser@qu.edu.qa}
\author[O. Rainio]{Oona Rainio}
\address{Department of Mathematics and Statistics, University of Turku, FI-20014 Turku, Finland}
\email{ormrai@utu.fi}
\author[M. Vuorinen]{Matti Vuorinen}
\address{Department of Mathematics and Statistics, University of Turku, FI-20014 Turku, Finland}
\email{vuorinen@utu.fi}

\begin{abstract}
Given a compact connected set $E$ in the unit disk $\mathbb{B}^{2}$, we
give a new upper bound for the conformal capacity of the condenser
$(\mathbb{B}^{2}, E)$ in terms of the hyperbolic diameter $t$ of $E$.
Moreover, for $t>0$, we construct a set of hyperbolic diameter $t$ and
apply novel numerical methods to show that it has larger capacity than
a hyperbolic disk with the same diameter. The set we construct is called a Reuleaux triangle in hyperbolic geometry and it has
constant hyperbolic width equal to $t$.
\end{abstract}
\maketitle

\textbf{Data availability statement.}
All the data used in the research for this article was created with MATLAB codes available in GitHub at \url{github.com/mmsnasser/hypdiam}.

%%%%%%%%%%%%%%%%%%%%%%%%%%%%%%%%%%%%%%%%%
%%%%%%%%%%%%%%%%%%%%%%%%%%%%%%%%%%%%%%%%%
%%%%%%%%%%%%%%%%%%%%%%%%%%%%%%%%%%%%%%%%%
\section{Introduction}
%%%%%%%%%%%%%%%%%%%%%%%%%%%%%%%%%%%%%%%%%
%%%%%%%%%%%%%%%%%%%%%%%%%%%%%%%%%%%%%%%%%
%%%%%%%%%%%%%%%%%%%%%%%%%%%%%%%%%%%%%%%%%
One of the famous problems of geometry is the problem of maximizing the volume of a geometric body given its surface area. This \emph{isoperimetric problem} is a constrained extremal problem connecting two domain functionals, the volume and the surface area of the domain in question. Other than this specific question, there are several kinds of constrained extremal problems motivated by geometry and mathematical physics that can be referred to as isometric problems. 

Already seventy years ago, G. P\'olya  and G. Szeg\"o studied isoperimetric problems in their famous book \cite{ps}, which inspired numerous later authors. They specifically devoted a lot of attention to isoperimetric problems involving condenser capacity. Condenser capacities are important tools in the study of
partial differential equations, Sobolev spaces, integral inequalities, potential theory,
see V. Maz\'{}ya \cite{m}
and J. Heinonen, T. Kilpel\"ainen, and O. Martio \cite{hkm}. Since the extremal situations for the isoperimetric problems often reflect symmetry, various symmetrization procedures can be used as a method for analysing isoperimetric problems, see A. Baernstein \cite{bae}. Furthermore,  in his pioneering work \cite{du}, V.N. Dubinin systematically developed capacity related methods
and  gave numerous applications of condenser capacities and symmetrization methods to classical
function theory.  Capacity is also one of the key techniques in the theory of quasiconformal and
quasiregular maps in the plane and space \cite{GMP,HKV,res,rick}.

An open, connected and non-empty set $G$ is called {\it a domain} and if $E\subset G$ is a compact non-empty set, then the pair $(G,E)$ is {\it a condenser}. {\it The capacity} of this condenser is defined by
\begin{align}\label{def_condensercapacity}
{\rm cap}(G,E)=\inf_u\int_{G}|\nabla u|^n dm,
\end{align}
where the infimum is taken over the set of all $C^\infty_0(G)$ functions $u: G\to[0,\infty)$ with
$u(x) \ge 1$ for all $x \in E$ and $dm$ is the $n$-dimensional Lebesgue measure. Below, we often choose $n=2$ and focus on the special case where $G\subsetneq\R^2$ is simply connected and $E$ is a continuum.

By classical results the capacity decreases under a geometric transformation called
symmetrization  \cite[Ch 6, p.215]{bae}, \cite{du}, \cite[Thm 5.3.11]{GMP},
\begin{align}\label{cap_symlowerbound}
{\rm cap}(G,E)\geq{\rm cap}(G_s,E_s),    
\end{align}
where $(G_s,E_s)$ is the condenser obtained by one of the well-known symmetrization procedures, such as the spherical symmetrization or the Steiner symmetrization. While finding the explicit formula for the capacity \eqref{def_condensercapacity} is usually impossible, the lower bound \eqref{cap_symlowerbound} can be often estimated or given explicitly \cite{du}, \cite[pp.180-181]{GMP}, \cite[Chapter 9]{HKV}. 

Our aim in this paper is to find upper bounds for the condenser capacity, when $n=2$, 
$G$ is a simply connected domain, and $E$ is a connected compact set. Numerous bounds are given in the literature in terms of domain functionals, such as the area of $G$, the diameter of $E$ and the distance from $E$ to the boundary $\partial G$ \cite{du,GMP,HKV,m,cchp,res,rick}. While these kinds of bounds have numerous applications as shown in the cited sources, these bounds do not reflect the conformal invariance of ${\rm cap}(G,E)$. We apply the conformally invariant hyperbolic metric in this paper and therefore our main results are conformally invariant.

By the conformal invariance of the capacity and the hyperbolic metric, we may assume without loss of generality that the domain $G$ is the unit disk $\B^2$ in the two-dimensional plane $\C=\R^2$. Naturally, we use here the  Riemann mapping theorem \cite{BM}. After this preliminary reduction, we look for upper bounds for the condenser capacity ${\rm cap}(\B^2,E)$ in terms of the hyperbolic metric $\rho_{\B^2}$ of the unit disk, when the hyperbolic diameter $\rho_{\B^2}(E)$ of the  compact set $E$ is fixed.

A first guess might be that for a compact set $E \subset \B^2\,,$ a majorant for ${\rm cap}(\B^2,E)$ would be the capacity of a hyperbolic disk with the hyperbolic diameter equal to that of $E$. This guess is motivated by a measure-theoretic isodiametric inequality, see Remark \ref{rmk_1guess}. However, the main result of this paper
is to show that this guess is wrong. For this purpose, we introduce the so-called hyperbolic Reuleaux triangle, which is a set of constant hyperbolic width, and then compute its conformal 
capacity with novel computational methods \cite{LSN17,Nas-ETNA,nvs,nv} to confirm our claim.
% However, we will demonstrate below in  terms of a set with constant hyperbolic width, 
%the so-called hyperbolic Reuleaux triangle, that this guess is wrong.
A valid upper bound for $\capa(\B^2,E)$ in terms of the hyperbolic diameter
is instead naturally  given in terms of the capacity of the minimal hyperbolic disk 
containing the set $E$ and here we apply the  work of  B.V. Dekster in \cite{de}
who found this  minimal radius.

\begin{theorem}\label{thm_1.3}
For a continuum $E\subset\B^2$ with the hyperbolic diameter equal to $t>0$, the inequality
\begin{align}\label{ine_in1.3}
{\rm cap}(\B^2,E) \leq\frac{2\pi}{\log((1+\sqrt{1+v^2})\slash v)}
\quad\text{with}\quad
v=(2/{\sqrt{3}}) \,{\rm sh}({t}/{2})\,. 
\end{align}
holds. 
\end{theorem}

Given a number $t>0\,,$ {\it the hyperbolic Jung radius} is the smallest number $r>0$ such that
every set $E \subset \B^2\,$ with the hyperbolic diameter equal to $t$ is contained in some hyperbolic disk with the radius equal to $r\,.$  Originally,
the Jung radius was found in the context of the Euclidean geometry \cite[p. 33, Thm 2.8.4]{mmo} and its
hyperbolic counterpart was found for dimensions $n\ge 2$ by B.V. Dekster in \cite{de}.
His result is formulated below as
Theorem \ref{thm_rjungE} and Theorem \ref{thm_1.3} is based on the special case $n=2$ of his work. It should be noticed that, by the Riemann mapping theorem, \eqref{ine_in1.3} directly applies to the case of planar simply connected domains.
The  sharp upper bound in Theorem \ref{thm_1.3} is not known and
this motivates the following open problem.

\begin{nonsec}{\bf Open problem.}\label{SolyninA}
Given $t>0$, identify all connected compact sets $E\subset \B^2$ with the hyperbolic diameter $t$, which maximize the capacity ${\rm cap}(\B^2,E)$.
\end{nonsec}

Theorem \ref{thm_1.3} provides an upper
bound for the quantity
\begin{equation}\label{mybfun}
b(t) \equiv \sup \{  {\rm cap}(\B^2,E) \,: E\,\, {\rm continuum\,\,and}\,\, \rho(E)=t\}\,
\end{equation}
that we will analyse further, in order to find a lower bound for it. To this end  we have
to apply numerical methods. Our first step is to write an algorithm for computing the hyperbolic diameter of a set in a simply connected domain. The  boundary integral equation method developed in a series of recent papers \cite{LSN17,Nas-ETNA,nvs,nv} is used. Using this method, we can compute the hyperbolic diameter and the capacity of a subset bounded by piecewise smooth curves in a polygonal domain or in the unit disk.
We show that the capacity of a hyperbolic disk with
diameter $t,$ denoted $b_1(t)\,,$  is a minorant for the above function $b(t)\,,$ i.e. $  b(t) \ge b_1(t)\,,$ see \eqref{Rbds}. For this purpose we introduce  the aforementioned hyperbolic Reuleaux triangle and our numerical work shows that its capacity majorizes the capacity of a disk with the same hyperbolic diameter.
A delicate point here is the essential role of the hyperbolic geometry: the hyperbolic Reuleaux triangle cannot be replaced by the Euclidean Reuleaux triangle with the same hyperbolic
diameter, for its capacity is not a majorant for $b_1(t)\,$ for $t>2\,.$ The numerical algorithm is of independent interest, because it enables one to experimentally study  the
hyperbolic geometry of planar simply connected polygonal domains.

We apply our result to quasiconformal maps and prove the following result.

\begin{theorem} \label{qcHypDia}
Let $f: G_1\to G_2 = f(G_1)$ be a $K$- quasiconformal homeomorphism between two
simply connected domains $G_1$ and $G_2$ in $\R^2\,,$ and let $E \subset G_1$ be
a continuum. Then
\begin{equation}  \label{qcHypDia2}
{\rm th} \frac{\rho_{G_{2}}(f(E))}{2} \le 4 \left({\rm th} \frac{h(2,\rho_{G_{1}}(E))}{2} \right)^{1/K}
\end{equation}
where $\rho_{G_1}$ and $\rho_{G_2}$ refer to the hyperbolic metrics of $G_1$ and $G_2\,,$resp., and $h(2,t)$ stands for the hyperbolic Jung radius of a set with the hyperbolic
diameter equal to $t\,$ defined in Theorem \ref{thm_rjungE} due to B.V. Dekster  \cite{de}.
\end{theorem}

For a large class of simply connected plane domains, so called $\varphi$-uniform domains,
we give explicit  bounds for the hyperbolic Jung radius of a compact set in the domain.

We are indebted to Prof. Alex Solynin for pointing out the above open problem.
%%%%%%%%%%%%%%%%%%%%%%%%%%%%%%%%%%%%%%%%%
%%%%%%%%%%%%%%%%%%%%%%%%%%%%%%%%%%%%%%%%%
%%%%%%%%%%%%%%%%%%%%%%%%%%%%%%%%%%%%%%%%%
\section{Preliminaries}
%%%%%%%%%%%%%%%%%%%%%%%%%%%%%%%%%%%%%%%%%
%%%%%%%%%%%%%%%%%%%%%%%%%%%%%%%%%%%%%%%%%
%%%%%%%%%%%%%%%%%%%%%%%%%%%%%%%%%%%%%%%%%

An open ball defined with the Euclidean metric is $B^n(x,r)=\{y\in\R^n\text{ }\text{: }|x-y|<r\}$ and the corresponding closed ball is $\overline{B}^n(x,r)=\{y\in\R^n\text{ }\text{: }|x-y|\leq r\}$. The sphere of these balls is $S^{n-1}(x,r)=\{y\in\R^n\text{ }\text{: }|x-y|=r\}$. Note that if the center $x$ or the radius $r$ is not otherwise specified in these notations, it means that $x=0$ and $r=1$. In a metric space $(X,d)\,,$ a ball centered at $x$ and with radius $r>0$ is  $B_d(x,r)\,$ and the diameter of a non-empty set $A \subset X$ is
$d(A).$

In 
%the upper half-plane $\uhp^n=\{(x_1,...,x_n)\in\R^n\text{: }\text{}x_n>0\}$ and 
the Poincar\'e unit ball $\B^n=\{x\in\R^n\text{:}|x|<1\}$, the hyperbolic metric is defined as \cite[(2.8) p. 15]{BM}
\begin{align*}
%\text{ch }\rho_{\uhp^n}(x,y)&=1+\frac{|x-y|^2}{2x_ny_n},\quad x,y\in\uhp^n,\\
\text{sh}^2\frac{\rho_{\B^n}(x,y)}{2}&=\frac{|x-y|^2}{(1-|x|^2)(1-|y|^2)},\quad x,y\in\B^n.
\end{align*}
\begin{comment}
In the special case $n=2$, this formula can be written as
\begin{align*}
%\text{th}\frac{\rho_{\uhp^2}(x,y)}{2}&=\text{th}\left(\frac{1}{2}\log\left(\frac{|x-\overline{y}|+|x-y|}{|x-\overline{y}|-|x-y|}\right)\right)=\left|\frac{x-y}{x-\overline{y}}\right|,\\
\text{th}\frac{\rho_{\B^2}(x,y)}{2}&=\text{th}\left(\frac{1}{2}\log\left(\frac{|1-x\overline{y}|+|x-y|}{|1-x\overline{y}|-|x-y|}\right)\right)=\left|\frac{x-y}{1-x\overline{y}}\right|=\frac{|x-y|}{A[x,y]},
\end{align*}
where $\overline{y}$ is the complex conjugate of $y$ and $A[x,y]=\sqrt{|x-y|^2+(1-|x|^2)(1-|y|^2)}$ is the Ahlfors bracket \cite[(3.17) p. 39]{HKV}.
\end{comment}
The hyperbolic segment between the points $x,y$ is denoted by $J[x,y]$. Furthermore, the hyperbolic balls $B_\rho(q,R)$ are Euclidean balls with
the center and the  radius given by the following lemma.

\begin{lemma}\label{lem_rhoball}\emph{\cite[(4.20) p. 56]{HKV}}
The equality $B_\rho(q,R)=B^n(j,h)$ holds for $q \in \B^n$ and $R > 0$, if 
\begin{align*}
j=\frac{q(1-t^2)}{1-|q|^2t^2},\quad
h=\frac{(1-|q|^2)t}{1-|q|^2t^2}\quad\text{and}\quad
t={\rm th}({R/2}).
\end{align*}
\end{lemma}

For a given simply connected planar domain $G\,,$ by means of the Riemann mapping theorem, one can define a conformal map of $G$ onto the unit disk $\B^2$, $f\,:\,G\to \B^2=f(G)$, and thus define the hyperbolic metric $\rho_G$ in $G$ by \cite{BM}
\begin{equation}\label{eq:rhoG}
\rho_G(x,y) = \rho_{\B^2}(f(x),f(y)),\quad x,y\in G.
\end{equation}
As an example, consider the simply connected domain $G$ inside the polygon with the vertices $0$, $3$, $3+\i$, $2+\i$, $2+0.2\i$, $1+0.2\i$, $1+\i$, and $\i$. Figure~\ref{fig:hyb-cir} (left) displays examples of hyperbolic circles in the domain $G$.
These hyperbolic circles are plotted by plotting the contour lines of the function
\[
u(z)=\rho_G(\alpha,z), \quad z\in G,
\]
corresponding to the levels (the hyperbolic radii of the hyperbolic circles) $0.5$, $1.5$, $4$, $10$, $16$, $21$, $23$, $23.5$, and $23.75$ where $\alpha=0.5+0.5\i$. The values of $\rho_G(\alpha,z)$ are computed using the method described in Appendix~\ref{sec:num-dia} with $n=2^{13}$.

\begin{figure}[hbt] %
\centerline{
\scalebox{0.525}{\includegraphics[trim=3.25cm 10.25cm 3.25cm 10.5cm,clip]{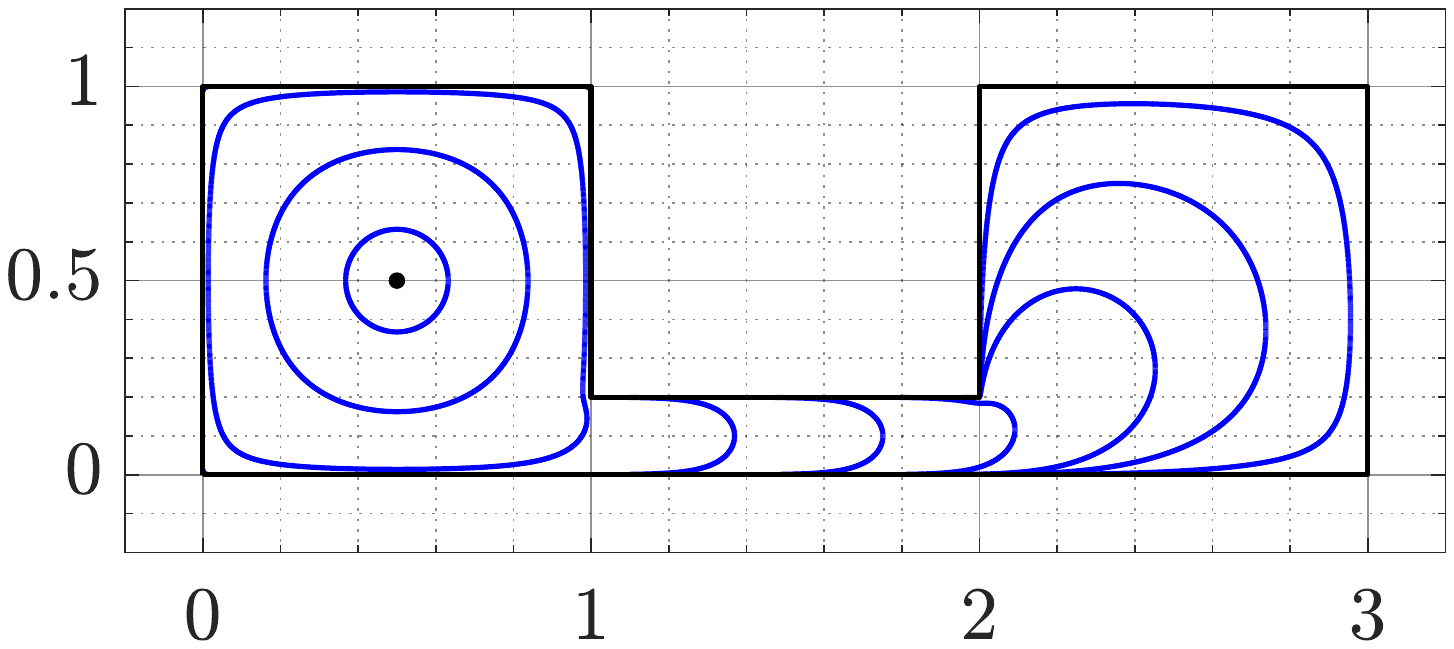}}
\hfill
\scalebox{0.525}{\includegraphics[trim=3.25cm 10.75cm 3.25cm 10.5cm,clip]{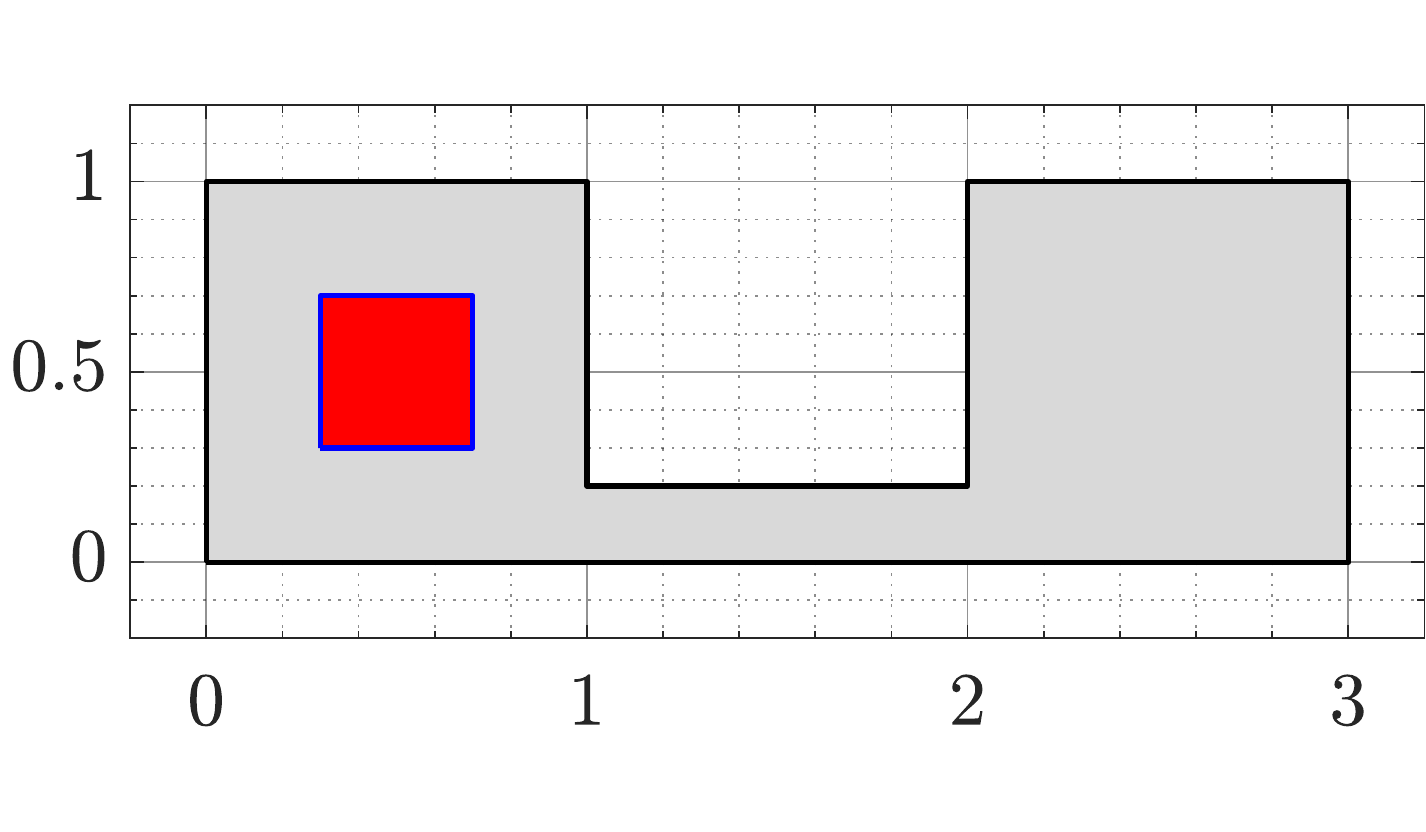}}
}
\caption{Examples of hyperbolic circles in a simply connected polygonal domain $G$ (left) and a square with the vertices $0.5\pm h+(0.5\pm h)\i$ for $h=0.2$ in the simply connected polygonal domain $G$ (right).}
\label{fig:hyb-cir}
\end{figure}

The hyperbolic diameter of a compact set $E \subset G$, denoted by $\rho_G(E)\,$, is defined by
\[
\rho_G(E) = \sup\{\rho_{G}(x,y)\,|\, x,y\in E\}.
\]
For the polygonal domain $G$ in Figure~\ref{fig:hyb-cir} (left), let $E\subset G$ be the closure of the square with the vertices $0.5\pm h+(0.5\pm h)\i$ for $0<h<0.5$ (see Figure~\ref{fig:hyb-cir} (right) for $h=2$). 
The approximate values of the hyperbolic diameter of the set $E$ for several values of $h$, computed by the method described in Appendix~\ref{sec:num-dia} with $\alpha=0.5+0.5\i$ and $n=2^{13}$, are given in Table~\ref{tab:hyp-E}. 
Table~\ref{tab:hyp-E} also presents the values of the capacity of the condenser $(G,E)$, which are computed using the method described in Appendix~\ref{sec:num-cap} with $\alpha=1.5+0.1\i$, $z_2=0.5+0.5\i$, and $n=2^{13}$.

%\begin{figure}[hbt] %
%\centerline{
%\scalebox{0.75}{\includegraphics[trim=0cm 10.75cm 0cm 10.5cm,clip]{HypSqu_2}}
%}
%\caption{The square with the vertices $0.5\pm h+(0.5\pm h)\i$ for $h=0.2$ in the simply connected polygonal domain $G$.}
%\label{fig:hyb-E}
%\end{figure}

\begin{table}[hbt]
\caption{The hyperbolic diameter $\rho_G(E)$ and the capacity $\capa(G,E)$ for the sets $G$ and $E$ shown in Figure~\ref{fig:hyb-cir} (right).}
\label{tab:hyp-E}%
\begin{tabular}{l@{\hspace{0.5cm}}|@{\hspace{0.5cm}}l@{\hspace{0.5cm}}|@{\hspace{0.5cm}}l} \hline
$h$     & $\rho_G(E)$ & $\capa(G,E)$ \\ \hline
$0.1$   & $1.0729$    & $4.1331$ \\
$0.2$   & $2.3071$    & $7.5564$ \\
$0.3$   & $3.9596$    & $14.2096$\\
$0.4$   & $6.7393$    & $33.9643$\\
$0.45$  & $9.5123$    & $72.8330$\\
\hline
\end{tabular}
\end{table}

\begin{comment}
Next, define the \emph{absolute (cross) ratio} as in \cite[(3.10), p. 33]{HKV}
\begin{align*}
|a,b,c,d|=\frac{q(a,c)q(b,d)}{q(a,b)q(c,d)},\text{ } a,b,c,d\in\overline{\R}^n;\quad
|a,b,c,d|=\frac{|a-c||b-d|}{|a-b||c-d|},\text{ }a,b,c,d\in\R^n.
\end{align*}
Note that this ratio is conformally invariant. Furthermore, it can be used to define the hyperbolic metric, see \cite[p.72, Thm 5.2.7]{be}.
\end{comment}

Note that while we already defined the condenser capacity in \eqref{def_condensercapacity}, its definition can be also written as
\begin{align*}
{\rm cap}(G,E)=\M(\Delta(E,\partial G;G)),
\end{align*}
as in  \cite[Thm 5.2.3, p.164]{GMP}, \cite[Thm 9.6, p. 152]{HKV}. Here, $\Delta(E,F;G)$ stands for the family of all the curves in the set $G$ that have one end point in the set $E$ and another end point in $F$ \cite[p. 106]{HKV}. The definition and basic properties of the modulus $\M(\Gamma)$ of a curve family $\Gamma$ can be found in \cite[Ch. 7, pp. 103-131]{HKV}. We often use the fact that the capacity is, in the same way as the modulus, conformally invariant.

\begin{lemma}\label{cgqm_5.14}\emph{}
(1) If $0<a<b$ and $D=\overline{B}^n(b)\backslash B^n(a)$,
\begin{align*}
\M(\Delta(S^{n-1}(a),S^{n-1}(b);D))=\omega_{n-1}(\log({b}/{a}))^{1-n}.
\end{align*}
(2) If $R>0$ then for $x \in \B^n$ and $R>0$
\begin{align*}
\M(\Delta(S^{n-1},B_{\rho}(x,R);\B^n))=\omega_{n-1}(\log({1}/{{\rm th} (R/2)}))^{1-n} \,.
\end{align*}
Here, $\omega_{n-1}$ is the $(n-1)$-dimensional surface area of the unit sphere $S^{n-1}.$ In particular, $\omega_1=2\pi.$
\end{lemma}

\begin{proof}
(1) This is a well-known basic fact, see e.g. \cite[(7.3), p. 107]{HKV}.

(2) The value of the left hand side is independent of $x$ by the M\"obius invariance
of the modulus and of the hyperbolic metric  and hence we may assume that $x=0\,.$
By Lemma \ref{lem_rhoball}, $B_{\rho}(x,R) = B_{\rho}(0,R) =B^n(0,{\rm th} (R/2))$ and
hence the proof follows from part (1).
\end{proof}

The Gr\"otzsch and Teichm\"uller capacities are the following decreasing homeomorphisms $\gamma_n:(1,\infty)\to (0,\infty), $  $\tau_n:(0,\infty)\to (0,\infty)$ \cite[(7.17), p. 121]{HKV}:
\begin{align*}
\gamma_n(s)&=\M(\Delta(\overline{\B}^n,[se_1,\infty];\R^n)),\quad s>1,\\
\tau_n(s)&=\M(\Delta([-e_1,0],[se_1,\infty];\R^n)),\quad s>0,
\end{align*}

%\begin{align*}
%\gamma_n(s)&=\M(\Delta(\overline{\B}^n,[se_1,\infty];\R^n\setminus(\overline{\B}^n\cup[se_1,\infty]))),\quad s>1,\\
%\tau_n(s)&=\M(\Delta([-e_1,0],[se_1,\infty];\R^n\setminus([-e_1,0]\cup[se_1,\infty]))),\quad s>0.
%\end{align*}
\noindent
where the notation $e_1,...,e_n$ stands for the unit vectors of $\R^n$. These capacities satisfy 
$\gamma_n(s) = 2^{n-1} \tau_n(s^2-1), $ for $s>1$ and various estimates are given in \cite[Chapter 9]{HKV} for $n\ge3\,.$ For $n=2, r\in(0,1),$ the following explicit
formulas are given by \cite[(7.18), p. 122]{HKV}, 
\begin{equation} \label{capGro}
\gamma_2(1/r)=\frac{2\pi}{\mu(r)}\,; \quad \mu(r)=\frac{\pi}{2}\frac{\K(\sqrt{1-r^2})}{\K(r)},\quad
\K(r)=\int^1_0 \frac{dx}{\sqrt{(1-x^2)(1-r^2x^2)}} \,.
\end{equation}
%for $s>1$, where
%\begin{align*}
%\mu(r)=\frac{\pi}{2}\frac{\K(\sqrt{1-r^2})}{\K(r)},\quad
%\K(r)=\int^1_0[(1-x^2)(1-r^2x^2)]^{-1\slash2}dx
%\end{align*}
%for $0<r<1$.

\begin{lemma}\label{lem_capgamma}
(1) \cite[Lemma 9.20, p. 163]{HKV}
If $x,y\in\B^n, x\neq y\,,$ and $E\subset\B^n$ is a continuum with $x,y\in E$, then 
\begin{align*}
\capa(\B^n,E)
\geq
\gamma_n\left(\frac{1}{\tnh(\rho_{\B^n}(x,y)\slash2)}\right).
\end{align*}
Here, the equality holds if $E$ is the geodesic segment $J[x,y]$ of the hyperbolic
metric joining $x$ and $y\,.$

(2) If $G$ is a simply connected domain in $\R^2\,,$ $E\subset G$ is a continuum, and
$x,y \in G, x\neq y\,,$ then
\begin{align*}
\capa(G,E)
\geq
\gamma_2\left(\frac{1}{\tnh(\rho_{G}(x,y)\slash2)}\right).
\end{align*}
\end{lemma}

\begin{proof}
(2) By the Riemann mapping theorem, we may assume without loss of generality
that $G=\B^2$ and hence the proof follows from part (1).
\end{proof}

\begin{nonsec}{\bf Sets of constant width \cite{mmo}.}  {\rm
Let $E \subset \R^n$ be a compact set with diameter equal to $t > 0\,.$
We say that $E$ is a set of constant width if for every $z  \in \partial E$,
\[
t= \sup \{  |z-x|:  x \in E \}\,.
\]}
\end{nonsec}

%\begin{figure}[hbt] %
%\centerline{
%\scalebox{0.5}{\includegraphics[width=\textwidth]{hypReuTriHlp}}
%}
%\caption{A hyperbolic Reuleaux triangle with vertices of $\{z:  |z|=0.5\}$.
%The dotted circular arc is part of the boundary of the disk 
%$B_{\rho}(0.5,M)$ defining this triangle and $M$ is the hyperbolic distance %between vertices.}
%\label{fig:Rtri}
%\end{figure}

\begin{figure}[hbt] %
\centerline{
\scalebox{0.5}{\includegraphics[trim=1cm 8.6cm 1.5cm 8.4cm,clip]{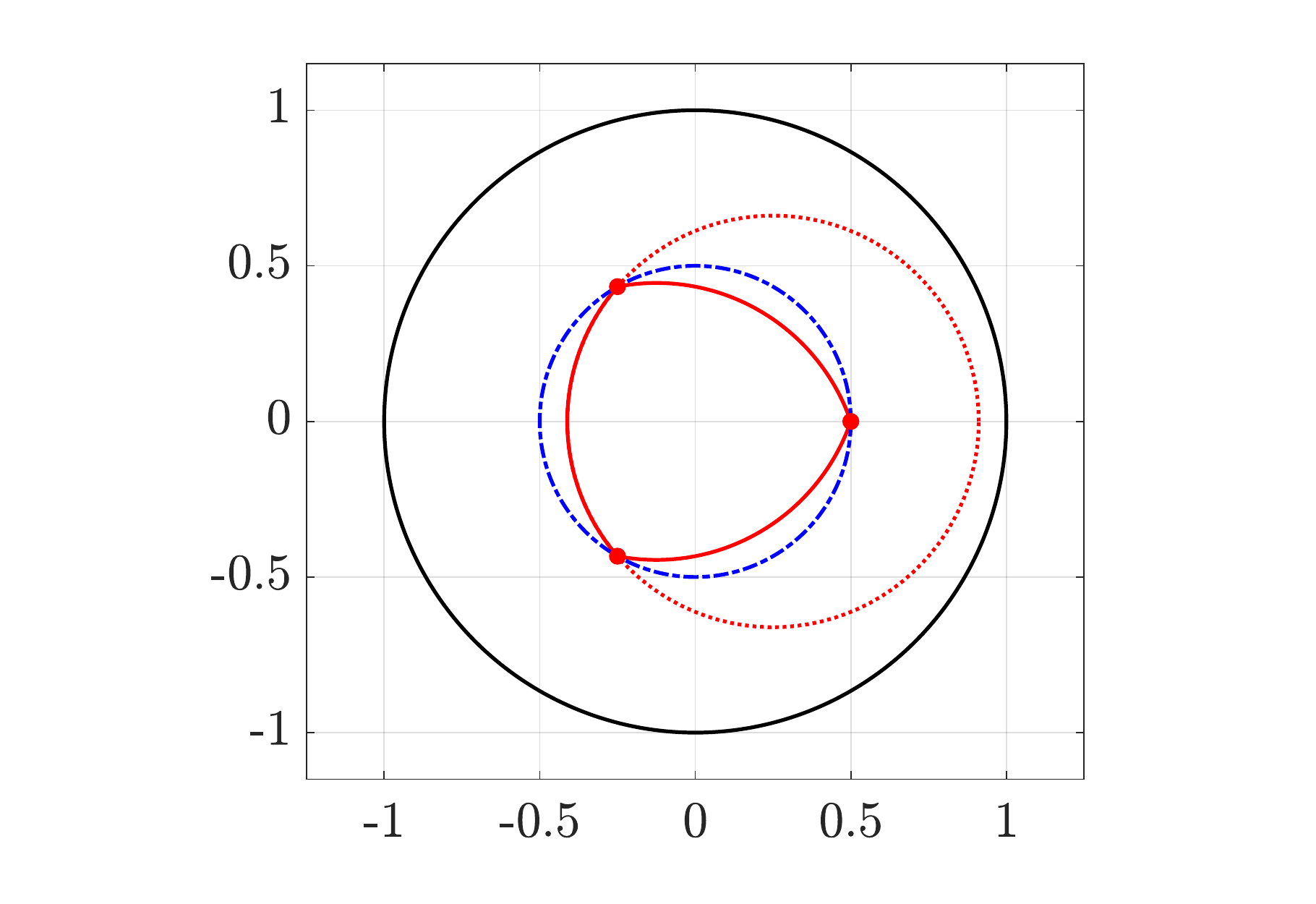}}
}
\caption{A hyperbolic Reuleaux triangle with vertices on $\{z:  |z|=0.5\}$.
The dotted circular arc is a part of the boundary of the disk 
$B_{\rho}(0.5,M)$, which is one of the three hyperbolic disks defining this triangle, and $M$ is the hyperbolic distance between vertices of the triangle.}
\label{fig:Rtri}
\end{figure}

\begin{nonsec}{\bf  The Euclidean and hyperbolic  Reuleaux triangle.} {\rm
An example of a set of constant width is {\it the Reuleaux triangle}, the intersection of three closed
disks with radii equal to $t>0$ and with centers at the vertices of an equilateral
triangle having side lengths equal to $t.$
We can define {\it the hyperbolic Reuleaux triangle}, a subset of the unit disk $\B^2$ in the same
way. To be more explicit, consider the  hyperbolic Reuleaux triangle with vertices at $r, r \,{\rm exp}(2 \pi \i /3),   r\, {\rm exp}(4 \pi \i/3)\,,$ and let
\[
M= \rho_{\B^2}(r, r \,{\rm exp}(2 \pi \i /3)) = 2 \,{\rm arsh}\, \frac{r \sqrt{3}}{1-r^2}\,.
\]
By Lemma \ref{lem_rhoball}, $B_{\rho}(r, M)=B^2(y,h)$ where $y$ and $h$ are given by
\[
y=\frac{(1-t^2)r}{1- r^2 t^2}, \quad h=\frac{(1-r^2)t}{1- r^2 t^2}\,\,,\quad t= {\rm th(arsh }\, \frac{r \sqrt{3}}{1-r^2})\quad .
\]
Let $D_1= \overline{B}^2(y,h)$ and let $D_2, D_3$ be the disks obtained from $D_1$ by 
rotation around the origin with angles $2 \pi/3$ and $4 \pi/3\,,$ resp. Now, the hyperbolic 
Reuleaux triangle with vertices at the above points is $D_1 \cap D_2 \cap D_3 \,.$
%Definition of Reuleaux triangle, take a look at \cite{mmo}.
}
\end{nonsec}

%%%%%%%%%%%%%%%%%%%%%%%%%%%%%%%%%%%%%%%%%
%%%%%%%%%%%%%%%%%%%%%%%%%%%%%%%%%%%%%%%%%
%%%%%%%%%%%%%%%%%%%%%%%%%%%%%%%%%%%%%%%%%
\section{Capacity and Jung radius}
%%%%%%%%%%%%%%%%%%%%%%%%%%%%%%%%%%%%%%%%%
%%%%%%%%%%%%%%%%%%%%%%%%%%%%%%%%%%%%%%%%%
%%%%%%%%%%%%%%%%%%%%%%%%%%%%%%%%%%%%%%%%%

For a compact subset $E$ of a metric space $X$, the Jung radius is the least number $r>0$ such that, for some $x\in X$, $E$ is a subset of the closed ball centered at $x$ with the radius $r$ 
\cite{mmo}.
%\textbf{cite smth!!!!!1}. 
The metric space in our work will be the hyperbolic disk and we denote the hyperbolic Jung radius of the set $E$ by $r_{Jung}(E)$. Clearly, it follows from the conformal invariance of the hyperbolic metric that the Jung radius is conformally invariant. Because of the same reason, for every simply connected domain $G\subsetneq\R^2$ and all compact sets $E\subset G$, there exists $z\in G$ with
\begin{align}\label{ine_caphypJung}
{\rm cap}(G,E)\leq{\rm cap}(G,\overline{B}_\rho(z,r_{Jung}(E))).     
\end{align}
By Lemma \ref{lem_rhoball}, $B_\rho(z,r_{Jung}(E))$ is conformally equivalent to $B^2(0,{\rm th}(r_{Jung}(E)\slash2))$ and thus it follows from \eqref{ine_caphypJung} and Riemann's mapping theorem that
\begin{equation} \label{2dimbd}
{\rm cap}(G,E)\leq\frac{2\pi}{\log(1\slash{\rm th}(r_{Jung}(E)\slash2))}.   
\end{equation}

\begin{theorem}\label{thm_rjungE}
{\rm B.V. Dekster \cite[Thm 2, (1.3)]{de}}
If $E\subset\B^n$ is a compact set with $\rho_{\B^n}(E) \le t$, then
\begin{align*}
r_{Jung}(E)\leq{\rm arsh}\left(\sqrt{\frac{2n}{n+1}}{\rm sh}\frac{t}{2}\right)\equiv h(n,t).    
\end{align*}
\end{theorem}

\begin{remark} Making use of the identity
\begin{align*}
{\rm th}\frac{M}{2}=\frac{{\rm sh}M}{1+\sqrt{1+{\rm sh}^2 M}}\,,\,\, M> 0\,, 
\end{align*}
we observe that
\begin{equation}\label{mysimple}
{\rm th}\frac{h(n,t)}{2}=\frac{u\, {\rm sh}(t/2)}{1+\sqrt{1+ u^2{\rm sh}^2 (t/2)}},\quad u =\sqrt{2n/(n+1)}\,\, , \quad t> 0\,.
\end{equation}
\end{remark}

For Lemma \ref{lem_ineqforhnt}, which provides bounds for the function $h(n,t)$, we first prove some preliminary results.

\begin{proposition}\label{lem_increasingh}
(1) For all $k>0$, the function $h:(0,\infty)\to\R$, $h(x)={\rm sh}(kx)\slash x$, is increasing.

(2) For all $x\geq1$, $k>0$, the inequality $x\,{\rm sh}k\leq{\rm sh}(kx)$ holds.

\end{proposition}
\begin{proof}
(1) Writing $f(x)={\rm sh}(kx)$ and $g(x)=x$, we see that $f'(x)\slash g'(x)=k{\rm ch}(kx)$ is increasing and, by \cite[Thm B.2, p. 465]{HKV}, so is $h(x)=f(x)\slash g(x)$. 

(2) Since the function $h(x)={\rm sh}(kx)\slash x$ of part (1) is increasing for all $k>0$,
\begin{align*}
x\geq1
\quad\Leftrightarrow\quad
h(1)\leq h(x)
\quad\Leftrightarrow\quad
{\rm sh}(k)\leq {\rm sh}(kx)\slash x
\quad\Leftrightarrow\quad
x\,{\rm sh}k\leq {\rm sh}(kx).
\end{align*}
\end{proof}

\begin{comment}

\begin{corollary}\label{cor_shkx}
For all $x\geq1$, $k>0$, the inequality $x\,{\rm sh}k\leq{\rm sh}(kx)$ holds.
\end{corollary}
\begin{proof}
Since the function $h(x)={\rm sh}(kx)\slash x$ of Proposition \ref{lem_increasingh} is increasing for all $k>0$,
\begin{align*}
x\geq1
\quad\Leftrightarrow\quad
h(1)\leq h(x)
\quad\Leftrightarrow\quad
{\rm sh}(k)\leq {\rm sh}(kx)\slash x
\quad\Leftrightarrow\quad
x\,{\rm sh}k\leq {\rm sh}(kx).
\end{align*}
\end{proof}
\end{comment}

\begin{lemma}\label{lem_ineqforhnt}
For all $n\geq2$, $t>0$, the inequality $\sqrt{2(n+1)\slash n}\leq t\slash h(n,t)\leq2$ holds.
\end{lemma}
\begin{proof}
We can write
\begin{align*}
&\sqrt{\frac{2(n+1)}{n}}\leq t\slash h(n,t)
\quad\Leftrightarrow\quad
h(n,t)={\rm arsh}\left(\sqrt{\frac{2n}{n+1}}{\rm sh}\frac{t}{2}\right)\leq t\sqrt{\frac{n}{2(n+1)}}\\
&\Leftrightarrow\quad
\sqrt{\frac{2n}{n+1}}{\rm sh}\frac{t}{2}\leq {\rm sh}\left(\frac{t}{2}\sqrt{\frac{2n}{n+1}}\right).
\end{align*}
By choosing $x=\sqrt{2n\slash(n+1)}\geq2\slash\sqrt{3}>1$ and $k=t\slash2>0$, we see that the inequality follows from Proposition \ref{lem_increasingh}. Furthermore, since
\begin{align*}
&t\slash h(n,t)\leq2
\quad\Leftrightarrow\quad
h(n,t)={\rm arsh}\left(\sqrt{\frac{2n}{n+1}}{\rm sh}\frac{t}{2}\right)\geq\frac{t}{2}
\quad\Leftrightarrow\quad
\sqrt{\frac{2n}{n+1}}{\rm sh}\frac{t}{2}\geq{\rm sh}\frac{t}{2}\\
&\Leftrightarrow\quad
\sqrt{\frac{2n}{n+1}}\geq1
\quad\Leftrightarrow\quad
n\geq1,
\end{align*}
the latter inequality in the lemma also holds.
\end{proof}

\begin{corollary} \label{JungBd}
(1) If $E$ is a compact subset of the unit ball $\B^n\,, n\ge 2\,,$ with the hyperbolic
diameter at most $t\,,$
then
\[
{\rm cap}(\B^n,E)\leq \frac{ \omega_{n-1}}{\left( \log (1/{\rm th} (h(n,t)/2))\right)^{n-1}}\,,
\quad t \le 2 h(n,t) \le t \sqrt{2n/(n+1)}\,.
\]

(2) If $E$ is a compact subset of a simply connected domain $G\subsetneq\R^2$, then
\begin{align*}
{\rm cap}(G,E)\leq\frac{2\pi}{\log((1+\sqrt{1+v^2})\slash v)}
\quad\text{with}\quad
v=\sqrt{\frac{4}{3}}{\rm sh}\frac{\rho_{G}(E)}{2}.    
\end{align*}
\end{corollary}

\begin{proof} (1) follows immediately from Theorem \ref{thm_rjungE}, Lemma \ref{lem_ineqforhnt} and Lemma \ref{cgqm_5.14} and some basic
properties of the modulus.

(2) The proof follows from \eqref{2dimbd} and the identity
 \eqref{mysimple}.
\end{proof}

\begin{nonsec}{\bf Proof of Theorem \ref{thm_1.3}.}
{\rm The proof follows from Corollary \ref{JungBd}(2). \hfill $\square$}
\end{nonsec}

\begin{figure}[hbt] %
\centerline{
\scalebox{0.5}{\includegraphics[trim=1cm 8.6cm 1.5cm 8.4cm,clip]{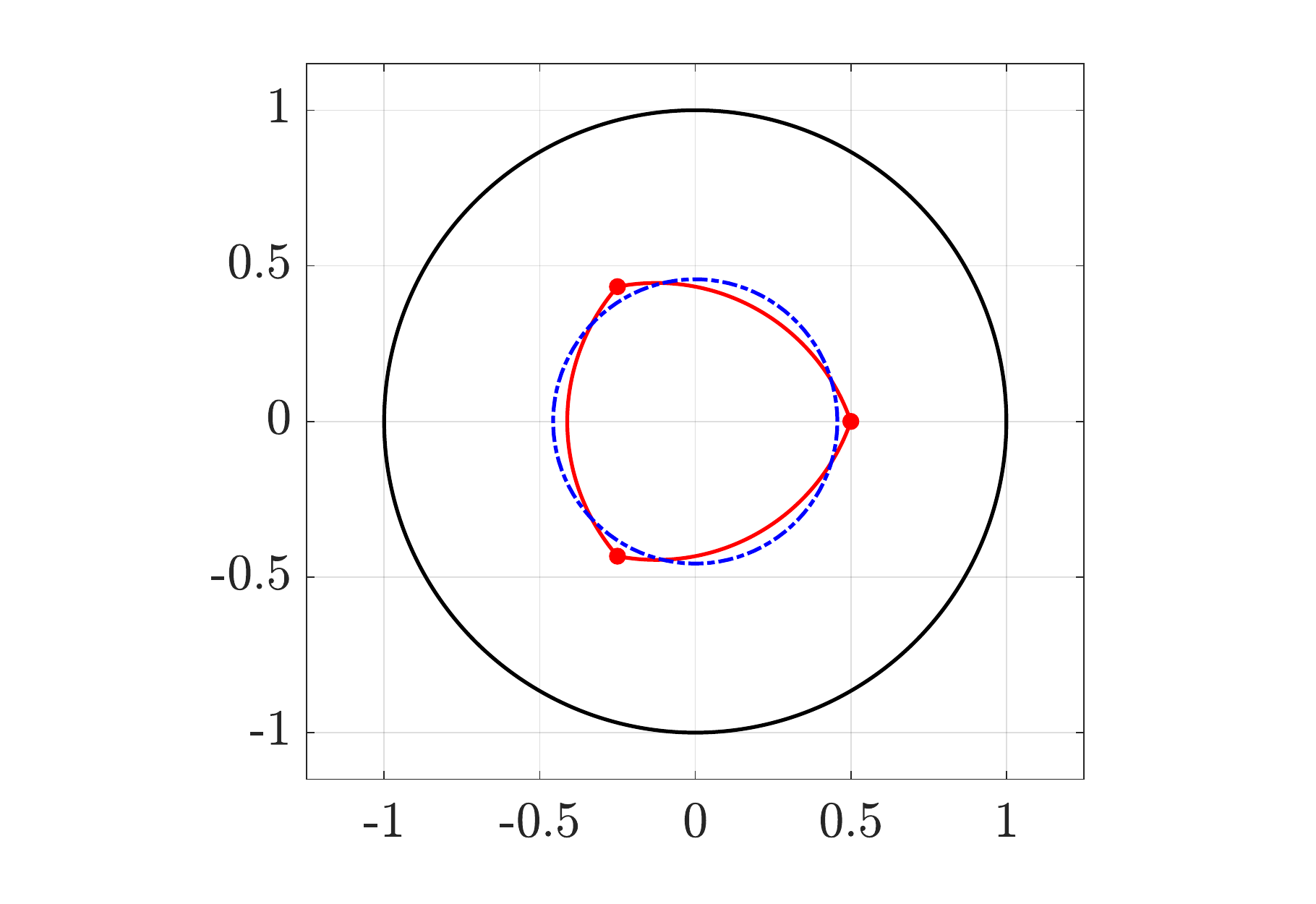}}
}
\caption{A hyperbolic Reuleaux triangle $T$ with vertices on $\{z:  |z|=0.5\}$
and a hyperbolic disk $D$ with the hyperbolic diameter equal to that of $T$.
Note that ${\rm cap}(\B^2, D) <{\rm cap}(\B^2, T) \,.$}
\label{fig:eqHypDia}
\end{figure}

\begin{remark}\label{compres}

We compare here the  capacities of several sets in terms of the hyperbolic diameter $t.$ The results are parametrized so that the vertices of the Reuleaux triangle
are on the circle $|z|=r\,.$ The results are given in the following table organized in seven columns
as follows: (1) $r=\tnh (t/2)$, (2) $t$, (3) $2 \pi/\mu(\tnh(t/2))$ i.e.
the capacity of the hyperbolic
geodesic segment of diameter $t$,
(4)  the capacity of a Euclidean Reuleaux triangle with hyperbolic diameter $t$,
(5)~$2 \pi/ \log(1/\tnh(t/4))$ i.e. the capacity of a hyperbolic
disk with diameter $t$, (6) the capacity of a hyperbolic Reuleaux triangle with diameter $t$,
(7) the upper bound given by Corollary \ref{JungBd}.
The values in columns (4) and (6) are computed using the method described in  Appendix~\ref{sec:num-cap} with $\alpha=0.4+0.6r$, $z_2=0$, and $n=3\times2^{8}$.

% Computed with hyp_Reuleaux_triangle_v6.m on 2020-10-06
\begin{table}[hbt]
\caption{The values of the computational results in Remark \ref{compres}.}
\label{tab:capac}%
\begin{tabular}{l@{\hspace{0.25cm}}|@{\hspace{0.25cm}}l@{\hspace{0.25cm}}|@{\hspace{0.25cm}}l@{\hspace{0.25cm}}|@{\hspace{0.25cm}}l@{\hspace{0.25cm}}|@{\hspace{0.25cm}}l@{\hspace{0.25cm}}|@{\hspace{0.25cm}}l@{\hspace{0.25cm}}|@{\hspace{0.25cm}}l} \hline
$r$    & h-diam   &  capSeg  &  capERtri &  capDisk  &  capHRtri &  capJung  \\ \hline
$0.05$ & $0.1734$ & $1.6396$ & $2.0242$  & $2.0017$  & $2.0245$  & $2.0974$ \\
$0.15$ & $0.5255$ & $2.3028$ & $3.1332$  & $3.0869$  & $3.1397$  & $3.3120$ \\
$0.25$ & $0.8937$ & $2.8457$ & $4.2040$  & $4.1470$  & $4.2360$  & $4.5324$ \\ \hline
$0.35$ & $1.2903$ & $3.3831$ & $5.4255$  & $5.3920$  & $5.5272$  & $5.9850$ \\
$0.45$ & $1.7305$ & $3.9583$ & $6.9289$  & $6.9994$  & $7.1957$  & $7.8687$ \\
$0.55$ & $2.2359$ & $4.6082$ & $8.8968$  & $9.2558$  & $9.5369$  & $10.5099$ \\
\hline
$0.65$ & $2.8416$ & $5.3821$ & $11.6482$ & $12.7508$ & $13.1588$ & $14.5855$ \\
$0.75$ & $3.6173$ & $6.3706$ & $15.8319$ & $18.9982$ & $19.6196$ & $21.8407$ \\
$0.85$ & $4.7413$ & $7.8018$ & $23.0155$ & $33.5301$ & $34.5948$ & $38.6613$ \\
$0.95$ & $7.0399$ & $10.7285$& $38.0667$ & $106.0995$& $108.9365$& $122.4953$ \\
\hline
\end{tabular}
\end{table}

\end{remark}

%\newpage
\begin{nonsec} {\bf Comparison of the bounds.} {\rm Corollary \ref{JungBd}(2)
gives an upper bound $b_2(t)$ for
\[
b(t) \equiv \sup \{  {\rm cap}(\B^2,E) \,: E\,\, {\rm connected\,\,and}\,\, \rho(E)=t\}\,.
\]
As we have seen above,
\[
b(t)  \ge 2\pi /\log(1/{\rm th}(t/4))\equiv b_1(t)\,.
\]
Define $b_2(t)$ as in Corollary \ref{JungBd}(2). Now, we know that for a hyperbolic Reuleaux triangle $T$ of hyperbolic
diameter equal to $t$ we have
\begin{equation} \label{Rbds}
b_1(t) \le {\rm cap}(\B^2,T)\le b_2(t)\,.
\end{equation}
Figure \ref{fig:cmpRtri} displays the graph of the function $b_2(t)/b_1(t)$
and its limit value $2/\sqrt{3}$ when $t \to \infty\,.$
}
\end{nonsec}

\begin{figure}[hbt] %
\centerline{
\scalebox{0.5}{\includegraphics[trim=1cm 8.6cm 1.5cm 8.5cm,clip]{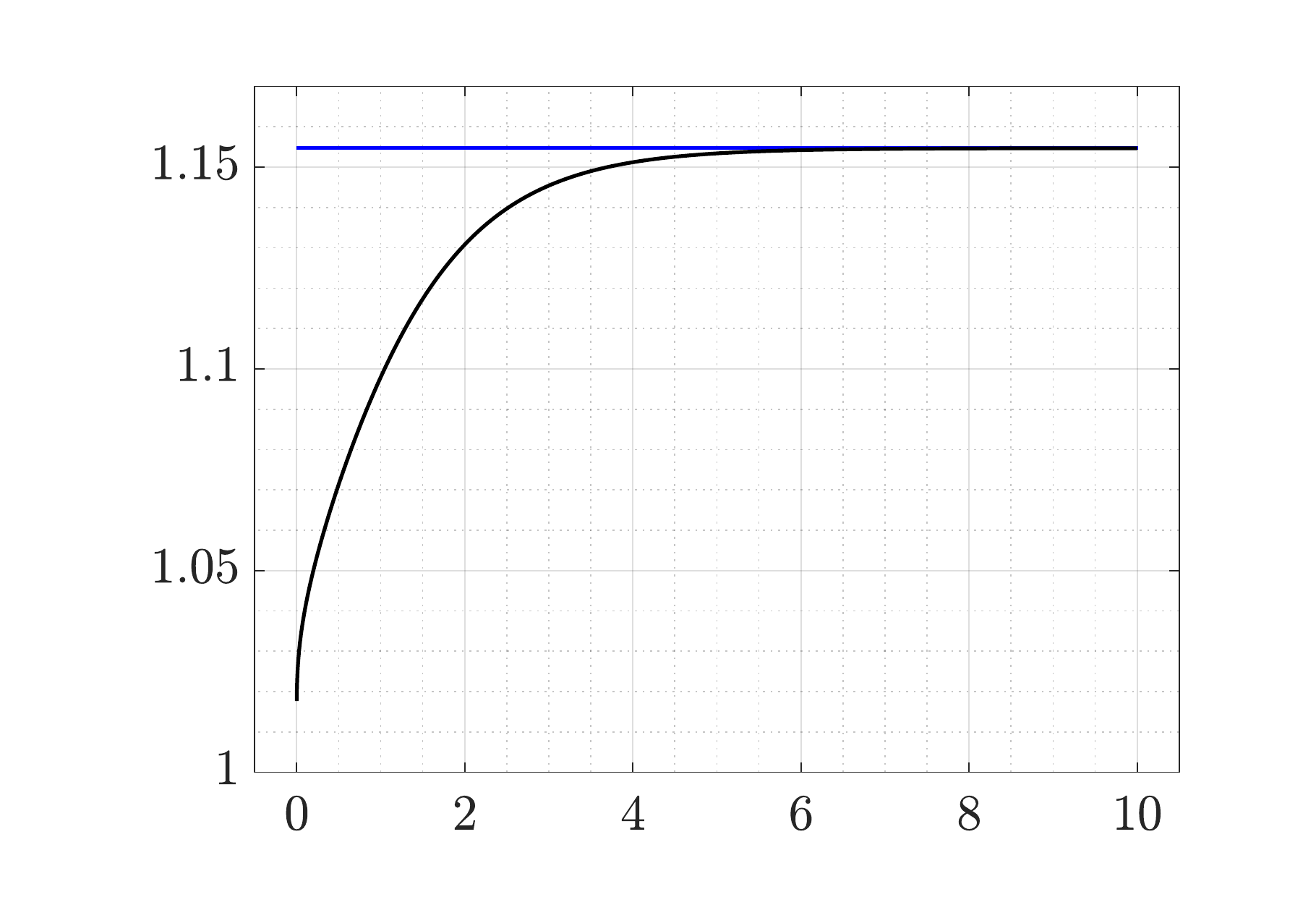}}
}
\caption{Graph of the quotient $b_2(t)/b_1(t)$ and $2/\sqrt{3}\,.$}
\label{fig:cmpRtri}
\end{figure}

\begin{nonsec}{\bf Capacity comparison:  the Euclidean vs  hyperbolic  Reuleaux triangle.} {\rm
We have shown above that the hyperbolic  Reuleaux triangle of diameter $t$ has a larger capacity
than $b_1(t)\,,$ the capacity of a hyperbolic disk with the same diameter.
A natural question is: Why do we use for this purpose the hyperbolic  Reuleaux triangle, not the
Euclidean one? It follows easily from Lemma \ref{lem_rhoball} that, as a point set,
the hyperbolic triangle contains the Euclidean one and thus has a larger capacity. The key point
now is that the capacity of the Euclidean Reuleaux triangle is smaller than $b_1(t)\,$
for $t >2\,.$
In Figure \ref{fig:hypVsEuc} (left) we demonstrate this fact by graphing, as a function of $t\,,$ the four quotients (1) Jung bound Corollary \ref{JungBd}(2) divided by  $b_1(t)$, (2) the capacity of the hyperbolic Reuleaux triangle/$b_1(t)$, (3)
$b_1(t)/b_1(t)$ (horizontal line), (4) the capacity of the Euclidean Reuleaux triangle/$b_1(t)\,.$
 
Figure \ref{fig:hypVsEuc} (right) displays three sets of equal hyperbolic diameter:
a disk, a hyperbolic Reuleaux triangle (solid line) and a Euclidean Reuleaux triangle
(dashed line).
}
\end{nonsec}

\begin{figure}[hbt] %
\centerline{
\scalebox{0.5}{\includegraphics[trim=3cm 8.5cm 3cm 8.5cm,clip]{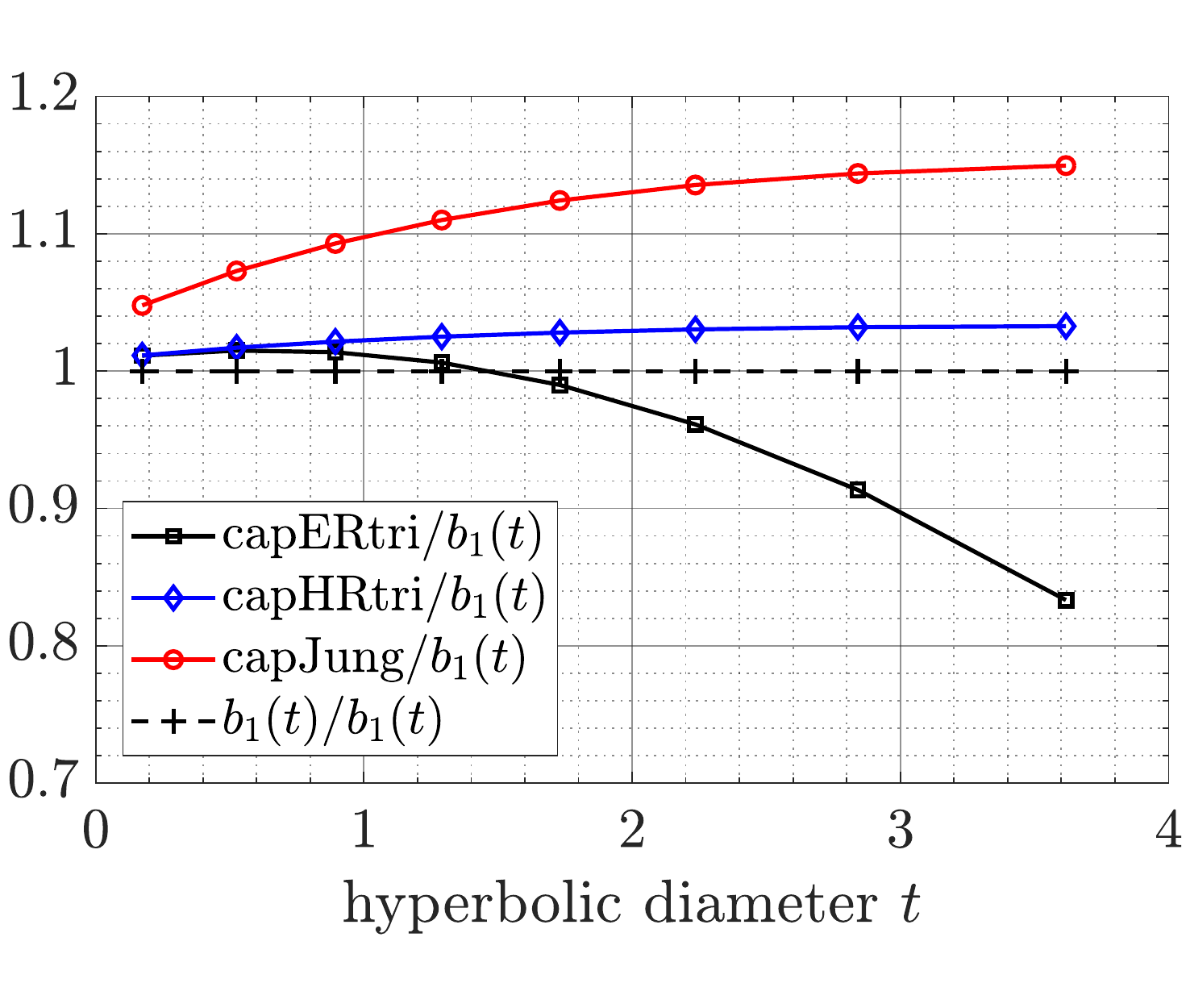}}
\hfill
\scalebox{0.5}{\includegraphics[trim=3cm 8.5cm 3cm 8.4cm,clip]{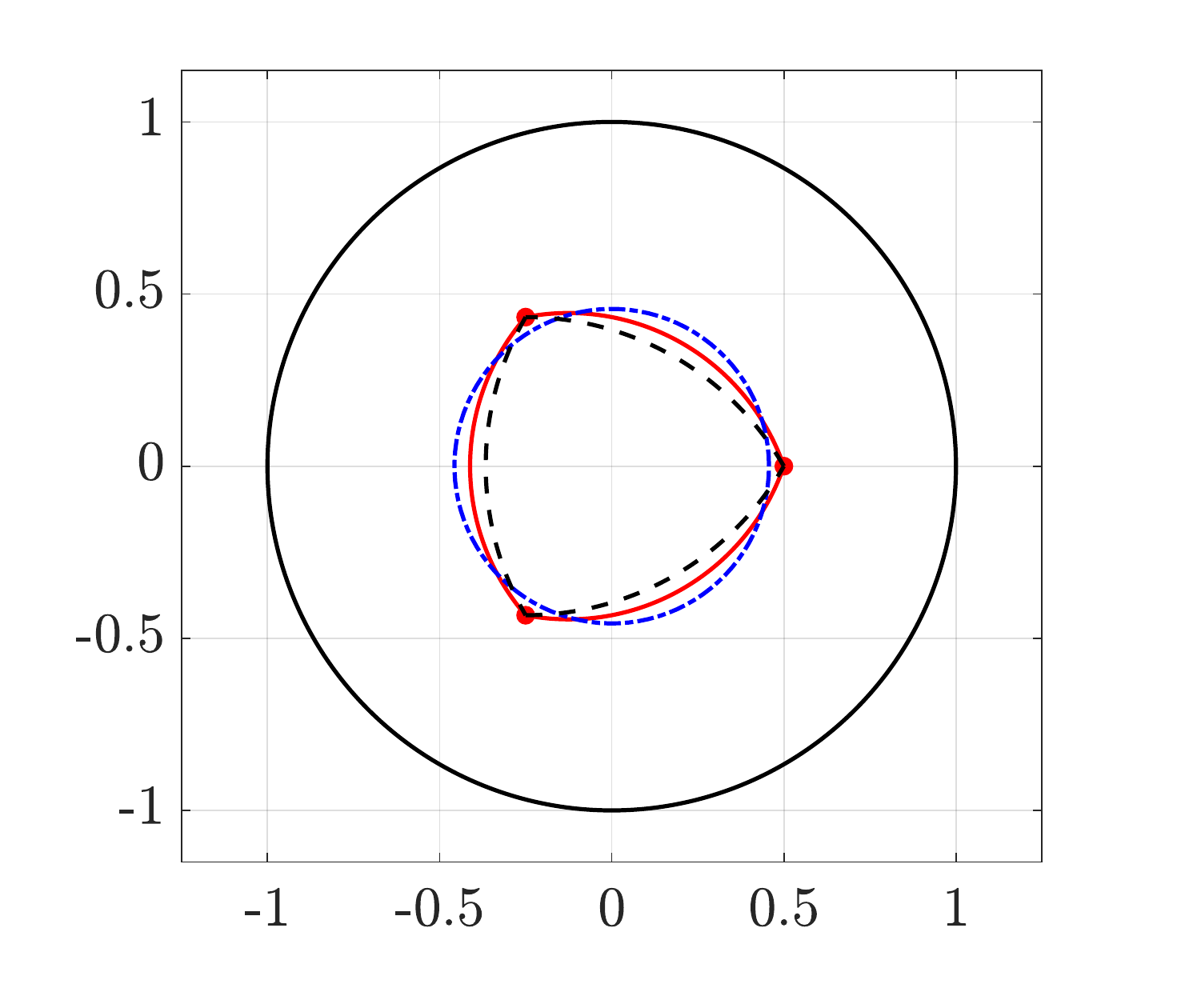}}
}
\caption{On the left, the capacities divided by $b_1(t)\,$ (see \eqref{Rbds}) as a function of the hyperbolic diameter $t$. On the right, three sets of equal hyperbolic diameter:
  a disk $D$ (dash-dotted line), a hyperbolic Reuleaux triangle $T$ (solid line), and a Euclidean Reuleaux triangle  $E$ (dashed line), $E \subset T\,.$
Note that by the results in
Table \ref{tab:capac}  we have ${\rm cap}(\B^2, E) <{\rm cap}(\B^2, D) <{\rm cap}(\B^2, T) \,$ for all large enough $t\,.$ 
\label{fig:hypVsEuc}}
\end{figure}

\begin{remark}\label{rmk_1guess}
For the Lebesgue measure of a measurable set $E\subset\R^n$, the well-known isodiametric
inequality states that $m(E)\le m(B^n(0,r))$ where the Euclidean diameter of $E$ is $2r$ \cite[p.548, Thm C.10]{Leo}. A similar result was proven very recently by K.J. B\"or\"oczky and \'A. Sagmeister in \cite{BS} for the balls in the hyperbolic geometry. As the above computational results demonstrate, for the condenser capacity there is no similar result.
\end{remark}

\begin{nonsec}{\bf Proof of Theorem \ref{qcHypDia}.}
{\rm 
Due to the conformal invariance of the hyperbolic metric, we may assume without loss of generality that $G_1=G_2=\B^2$. Let $\Delta=\Delta(E,\partial\B^2;\B^2)$ be the family of all curves in $\B^2$ joining $E$ and $\partial\B^2$. By quasiconformality,
\begin{align}\label{ine_InhypDiaProof1}
\M(f\Delta)\leq K\M(\Delta).    
\end{align}
Next, because $\mu(t)<\log(4\slash r)$ by \cite[(7.21)]{HKV} for $r\in(0,1)$, we obtain by Lemma \ref{lem_capgamma}(2) and \eqref{capGro} that
\begin{align}\label{ine_InhypDiaProof2}
\M(f\Delta)
\geq\gamma_2\left(\frac{1}{\text{th}(\rho_{G_2}(f(E))\slash2)}\right)
\geq\frac{2\pi}{\mu(\text{th}(\rho_{G_2}(f(E))\slash2))}
\geq\frac{2\pi}{\log(4\slash(\text{th}(\rho_{G_2}(f(E))\slash2)))}.
\end{align}
On the other hand, by Corollary \ref{JungBd},
\begin{align}\label{ine_InhypDiaProof3}
\M(f\Delta)
\leq\frac{2\pi}{\log(1\slash(\text{th}(h(2,\rho_{G_1}(E))\slash2)))}.
\end{align}
The inequalities \eqref{ine_InhypDiaProof1}, \eqref{ine_InhypDiaProof2} and \eqref{ine_InhypDiaProof3} together yield
\begin{align*}
\text{th}\frac{\rho_{G_2}(f(E))}{2}
\leq4\left(\text{th}\frac{h(2,\rho_{G_1}(E))}{2}\right)^{1\slash K},
\end{align*}
as desired.
\hfill $\square$}
\end{nonsec}

%%%%%%%%%%%%%%%%%%%%%%%%%%%%%%%%%%%%%%%%%
%%%%%%%%%%%%%%%%%%%%%%%%%%%%%%%%%%%%%%%%%
%%%%%%%%%%%%%%%%%%%%%%%%%%%%%%%%%%%%%%%%%
\section{Upper bounds for the hyperbolic Jung radius}
%%%%%%%%%%%%%%%%%%%%%%%%%%%%%%%%%%%%%%%%%
%%%%%%%%%%%%%%%%%%%%%%%%%%%%%%%%%%%%%%%%%
%%%%%%%%%%%%%%%%%%%%%%%%%%%%%%%%%%%%%%%%%

In  view of Corollary \ref{JungBd}, it is natural to look for bounds of the hyperbolic Jung radius of a compact set in a simply connected plane domain $G$. Perhaps a first question to study is whether we can find an upper bound in terms of the domain functional $d(E)\slash d(E,\partial G)$.  As Example \ref{ex51} demonstrates, this is not true in general simply connected domains, but by \eqref{jungPhi}
such a majorant is valid for $\varphi$-uniform domains.

\begin{example} \label{ex51}
For $G=\B^2\backslash[0,1)$, let $t\in(0,1\slash4)$, fix the points $x_t=(1\slash2,t)$, $y_t=(1\slash2,-t)$, and let $E_t$ be the set $\{x_t,y_t\}\subset G$. Then $d(E_t)\slash d(E_t,\partial G)=2$ but $\rho_G(x_t,y_t)\to\infty$ if $t\to0^+$. Therefore,  the hyperbolic Jung radius has no bounds in terms of $d(E)\slash d(E,\partial G)$.
\end{example}

\nonsec{\rm {\bf $\varphi$-uniform domains.}}
Let $\varphi:[0,\infty)\to[0,\infty)$ be an increasing homeomorphism and $G\subset\R^2$ a simply connected domain. We say that $G$ is $\varphi$-uniform if
\begin{align}
\rho_G(x,y)\leq\varphi\left(\frac{|x-y|}{\min\{d_G(x),d_G(y)\}}\right)    
\end{align}
for all $x,y\in G$. 

The class of $\varphi$-uniform domains  \cite[pp. 84-85]{HKV} contains many types of domains, including, for instance, all convex domains and so called quasidisks, which are images of the unit disk under quasiconformal maps of the plane \cite{gh}.

Now, we observe that if $E$ is a compact subset of a simply connected $\varphi$-uniform domain $G$, then by Theorem \ref{thm_rjungE},
\begin{align} \label{jungPhi}
r_{Jung}(E)\leq{\rm arsh}\left(\frac{2}{\sqrt{3}}{\rm sh}\frac{\varphi(d(E)\slash d(E,\partial G))}{2}\right).    
\end{align}  

Finally, we give a simple sufficient condition for a domain $G\subset\R^2$ to be $\varphi$-uniform: There exists $c\geq1$ such that every pair of points $x,y$ in $G$ can be joined by a curve $\gamma$ with length at most $c|x-y|$ so that
\begin{align*}
d(\gamma,\partial G)\geq(1\slash c)\min\{d_G(x),d_G(y)\}.   
\end{align*}
For more details, see \cite[p. 35]{gh} and \cite[pp. 84-85]{HKV}. 

\begin{rem} Recall that in every plane domain $G\,,$ the hyperbolic diameter of a continuum 
$E\subset G$ is bounded in terms of $d(E)/d(E,\partial G)$ \cite[6.32]{HKV} and hence so is its hyperbolic
Jung radius by Theorem \ref{thm_rjungE}.
\end{rem}

%\begin{figure}[hbt] %
%\centerline{
%\scalebox{0.6}{\includegraphics[trim=3cm 8.5cm 3cm 9cm,clip]{nonconvex}}
%}
%\caption{A nonconvex domain and a nonconvex compact subset.}
%\label{fig:nonconvex}
%\end{figure}

%%%%%%%%%%%%%%%%%%%%%%%%%%%%%%%%%%%%%%%%%
%%%%%%%%%%%%%%%%%%%%%%%%%%%%%%%%%%%%%%%%%
%%%%%%%%%%%%%%%%%%%%%%%%%%%%%%%%%%%%%%%%%
\appendix
\section{Computational methods}
\label{sec:num}
%%%%%%%%%%%%%%%%%%%%%%%%%%%%%%%%%%%%%%%%%
%%%%%%%%%%%%%%%%%%%%%%%%%%%%%%%%%%%%%%%%%
%%%%%%%%%%%%%%%%%%%%%%%%%%%%%%%%%%%%%%%%%

Computational tools for computing several conformal invariants in simply and doubly connected domains have been presented recently in~\cite{nvs,nv}. 
These tools are based on using the boundary integral equation with the generalized Neumann kernel. A fast numerical method for solving the integral equation is presented in~\cite{Nas-ETNA} which makes use of the Fast Multipole Method toolbox~\cite{Gre-Gim12}. 
In this appendix, we briefly describe these tools and demonstrate how they can be applied to compute numerically the hyperbolic diameter of compact sets as well as the capacity of condensers. 

\nonsec{\bf Numerical computation of hyperbolic diameter.}\label{sec:num-dia} 
%\begin{comment}
%The hyperbolic distance $\rho_{\D}(x,y)$ between any two points $x,y\in\D$ can be computed by \cite[(2.8) p. 15]{BM}
%\[
%\text{sh}^2\frac{\rho_{\D}(x,y)}{2}=\frac{|x-y|^2}{(1-|x|^2)(1-|y|^2)}.
%\]
%If $E\subset\D$, then the hyperbolic diameter of $E$ is defined by
%\[
%\mbox{h-diam}(E) = \sup\{\rho_{\D}(x,y)\,|\, x,y\in E\}.
%\]
%If $E$ is a compact set, then [do we need a ref]
%\begin{equation}\label{eq:diam-max-D}
%\mbox{h-diam}(E) = \max\{\rho_{\D}(x,y)\,|\, x,y\in \partial E\}.
%\end{equation}
%
%For a simply connected Jordan domain $G$ in the plane, the hyperbolic distance $\rho_{G}(x,y)$ between any two points $x,y\in G$ can be defined in terms of the conformal Riemann mapping function $f\,:\,G\to\D=f(G)$ by 
%\[
%\rho_{G}(x,y) = \rho_{\D}(f(x),f(y)).
%\]
%\end{comment}
If $E\subset G$ is a compact set in a simply connected domain $G$, then the hyperbolic diameter of $E$ with respect to $G$, 
\begin{equation}\label{eq:diam-max-G}
\rho_G(E) = \max\{\rho_{G}(x,y)\,|\, x,y\in \partial E\},
\end{equation}
can be computed once we have an algorithm for the hyperbolic distance $\rho_G(x,y)\,.$
In this paper, the maximum in~\eqref{eq:diam-max-G} is approximated numerically by discretizing the boundary $\partial E$ with a sufficiently large number of points. 
The hyperbolic distance $\rho_{G}(x,y)$ itself is approximated numerically by approximating a conformal mapping from the domain $G$ onto the unit disk $\B^2$.

A MATLAB function \verb|hypdist| for approximating the hyperbolic distance $\rho_{G}(x,y)$ when the boundary $\Gamma=\partial G$ is a piecewise smooth Jordan curve is presented in~\cite{nvs}. 
To use the function \verb|hypdist|, we parametrize $\Gamma$, which is assumed to be oriented counterclockwise, by a $2\pi$-periodic complex function $\eta(\delta(t))$, $t\in[0,2\pi]$, where $\delta\,:\,[0,2\pi]\to[0,2\pi]$ is a bijective strictly monotonically increasing function. When $\Gamma$ is smooth, we choose $\delta(t)=1$. For piecewise smooth boundary $\Gamma$, the function $\delta$ is chosen as described in~\cite[p.~697]{LSN17} (see also \cite{kre90}). We define $n$ equidistant nodes $s_1, \ldots, s_n$ in the interval $[0,2\pi]$ by
\begin{equation}\label{eq:s_i}
s_k = (k-1) \frac{2 \pi}{n}, \quad k = 1, \ldots, n,
\end{equation}
where $n$ is an even integer. Then we compute the vectors \texttt{et} and \texttt{etp} by
\[
\texttt{et}  = \eta(\delta(\bs))\in\C^{n}, \qquad
\texttt{etp} = \eta'(\delta(\bs))\delta'(\bs)\in\C^{n},  
\]
where $\bs=[s_1,\ldots, s_n]\in\R^n$. We also discretize the boundary of $E$ by a vector of points \verb|z|.
Then, the hyperbolic diameter $\rho_G(E)$ is approximated by calling
\begin{verbatim}
 max(max(hypdist(et,etp,n,alpha,z,z))),
\end{verbatim}
where $\alpha$ is  an auxiliary point in $G$.
For more details, we refer the reader to~\cite{nvs}.

\nonsec{\rm {\bf Numerical computation  of the capacity.}} \label{sec:num-cap}
Consider a bounded simply connected domain $G$ in the plane and a compact set $E\subset G$ such that $D= G\setminus E$ is a doubly connected domain. 
In this paper, the capacity of the condenser $(G,E)$ will be computed by the MATLAB function \verb|annq| from~\cite{nv}. 
The boundary components of $D= G\setminus E$ are assumed to be piecewise smooth Jordan curves.
Let $\Gamma_1$ be the external boundary component and $\Gamma_2$ be the inner boundary component such that $\Gamma_1$ is oriented counterclockwise and $\Gamma_2$ is oriented clockwise.
We parametrize $\Gamma_j$ by a $2\pi$-periodic complex function $\eta_j(\delta_j(t))$, $t\in[0,2\pi]$, where $\delta_j\,:\,[0,2\pi]\to[0,2\pi]$ is a bijective strictly monotonically increasing function,
$j=1,2$. When $\Gamma_j$ is smooth, we choose $\delta_j(t)=1$. For piecewise smooth boundary component $\Gamma_j$, the function $\delta_j$ is chosen as in~\cite[p.~697]{LSN17}. We compute the vectors \texttt{et} and \texttt{etp} by
\begin{eqnarray*}
\texttt{et} &=& [\eta_1(\delta_1(\bs))\,,\,\eta_2(\delta_2(\bs))]\in\C^{2n}, \\
\texttt{etp} &=& [\eta_1'(\delta_1(\bs))\delta_1'(\bs)\,,\,\eta_2'(\delta_2(\bs))\delta_2'(\bs)]\in\C^{2n},  
\end{eqnarray*}
where $\bs=[s_1,\ldots, s_n]\in\R^n$ and $s_1, \ldots, s_n$ are given by~\eqref{eq:s_i}. Then the MATLAB function \verb|annq| can be used to approximate $\capa(G,E)$ as follows,
\begin{verbatim}
  [~,cap] = annq(et,etp,n,alpha,z2,'b'),
\end{verbatim}
where $\alpha$ is an auxiliary point in the domain $D$ and $z_2$ is an auxiliary point in the interior of $E$ (see Figure~\ref{fig:hyb-cir} (right)).
The readers are referred to~\cite{nv} for more details.

The values of the parameters in the functions \verb|hypdist| and \verb|annq| are chosen as in~\cite{nvs,nv}.  
The codes for all presented computations in this paper are available in the link \url{https://github.com/mmsnasser/hypdiam}.

%%%%%%%%%%%%%%%%%%
%%%%%%%%%%%%%%%%%%
%%%%%%%%%%%%%%%%%%
%  hypdiabiblio.tex
%%%%%%%%%%%%%%%%%%
%%%%%%%%%%%%%%%%%%
%%%%%%%%%%%%%%%%%%
\def\cprime{$'$} \def\cprime{$'$} \def\cprime{$'$}
\providecommand{\bysame}{\leavevmode\hbox to3em{\hrulefill}\thinspace}
\providecommand{\MR}{\relax\ifhmode\unskip\space\fi MR }
% \MRhref is called by the amsart/book/proc definition of \MR.
\providecommand{\MRhref}[2]{%
  \href{http://www.ams.org/mathscinet-getitem?mr=#1}{#2}
}
\providecommand{\href}[2]{#2}

\end{document}